\documentclass[12pt]{article}
\usepackage[centertags]{amsmath}
\usepackage{amsfonts}
\usepackage{amssymb}
\usepackage{latexsym}
\usepackage{amsthm}
\usepackage{newlfont}
\usepackage{graphicx}
\usepackage{listings}
\usepackage{booktabs}
\usepackage{abstract}
\usepackage{xcolor}
\date{}

\newlength{\defbaselineskip}
\newcommand{\setlinespacing}[1]%
           {\setlength{\baselineskip}{#1 \defbaselineskip}}

\newcommand{\N}{{\mathbb{N}}}

\newcommand{\actaqed}{\hfill $\actabox$}
{\medskip\noindent \textit{Proof of #1. }}%
{\actaqed \medskip}

\def\C{{\mathcal C}}

\def \Tr{\mathcal T}

\def\R{{\mathbb R}}
\def\Z{\mathbb Z}

\def \T{\mathbb T}

\def\bbC{\mathbb C}
\def \<{\langle}
\def\>{\rangle}

\def \va{\varepsilon}
\def \ro{\varrho}

\def \sp{\operatorname{span}}

\def\bx{\mathbf x}
\def\by{\mathbf y}

\def\bk{\mathbf k}

\def\bw{\mathbf w}

\def\btt{\mathbf t}

\def\bW{\mathbf W}
\def\bH{\mathbf H}

\def\bF{\mathbf F}

\newtheorem{Theorem}{Theorem}[section]

\newtheorem{Remark}{Remark}[section]

\numberwithin{equation}{section}

\newcommand{\be}{\begin{equation}}
\newcommand{\ee}{\end{equation}}

\begin{document}

\title{On optimal recovery in $L_2$}

\author{V. Temlyakov\thanks{University of South Carolina, Steklov Institute of Mathematics, Lomonosov Moscow State University, and Moscow Center for Fundamental and Applied Mathematics.} } \maketitle

\begin{abstract}
	{We prove that the optimal error of recovery in the $L_2$ norm of functions from a class $\bF$ can be bounded above by the value of the Kolmogorov width of $\bF$ in the uniform norm. We demonstrate on a number of examples of $\bF$ from classes of functions with mixed smoothness that the obtained inequality provides a powerful tool for estimating errors of optimal recovery.}
\end{abstract}


\section{Introduction}
\label{Int}

The problem of recovery (reconstruction) of an unknown function defined on a subset of  $\R^d$ from its samples at a finite number of points is a fundamental problem of pure and applied mathematics. We would like to construct recovering operators (algorithms) which are good in the sense of accuracy, stability, and computational complexity. In this paper we discuss 
the issue of accuracy. Following a standard in approximation theory approach we define some 
optimal characteristics -- the Kolmogorov widths and errors of optimal recovery -- for a given function class and establish relations between them. We show that in the case of recovery in 
the $L_2$ norm the weighted least squares algorithms are reasonably good recovering methods. Our analysis is based on recent deep results in discretization of the $L_2$ norms of
functions from finite dimensional subspaces  (see \cite{VT158}, \cite{DPSTT2}, and \cite{LT}). 
We point out that the corresponding discretization results were obtained with a help of fundamental results from \cite{BLM}, \cite{BSS}, and \cite{MSS}. We now proceed to a formulation of the main result of the paper. Further discussion is given in Sections \ref{B} and \ref{C}.

Let $\Omega$ be a compact subset of $\R^d$ with the probability measure $\mu$. By $L_p$ norm, $1\le p< \infty$, of the complex valued function defined on $\Omega$,  we understand
$$
\|f\|_p:=\|f\|_{L_p(\Omega,\mu)} := \left(\int_\Omega |f|^pd\mu\right)^{1/p}.
$$
By $L_\infty$ norm we understand the uniform norm of continuous functions
$$
\|f\|_\infty := \max_{\bx\in\Omega} |f(\bx)|.
$$

Recall the setting 
 of the optimal recovery. For a fixed $m$ and a set of points  $\xi:=\{\xi^j\}_{j=1}^m\subset \Omega$, let $\Phi_\xi $ be a linear operator from $\bbC^m$ into $L_p(\Omega,\mu)$.
Denote for a class $\bF$ (usually, centrally symmetric and compact subset of $L_p(\Omega,\mu)$)
$$
\varrho_m(\bF,L_p) := \inf_{\text{linear}\, \Phi_\xi; \,\xi} \sup_{f\in \bF} \|f-\Phi_\xi(f(\xi^1),\dots,f(\xi^m))\|_p.
$$
The above described recovery procedure is a linear procedure. 
The following modification of the above recovery procedure is also of interest. We now allow any mapping $\Phi_\xi : \bbC^m \to X_N \subset L_p(\Omega,\mu)$ where $X_N$ is a linear subspace of dimension $N\le m$ and define
$$
\varrho_m^*(\bF,L_p) := \inf_{\Phi_\xi; \xi; X_N, N\le m} \sup_{f\in \bF}\|f-\Phi(f(\xi^1),\dots,f(\xi^m))\|_p.
$$

In both of the above cases we build an approximant, which comes from a linear subspace of dimension at most $m$. 
It is natural to compare quantities $\varrho_m(\bF,L_p)$ and $\varrho_m^*(\bF,L_p)$ with the 
Kolmogorov widths. Let $\bF\subset L_p$ be a centrally symmetric compact. The quantities  
$$
d_n (\bF, L_p) := \operatornamewithlimits{inf}_{\{u_i\}_{i=1}^n\subset L_p}
\sup_{f\in \bF}
\operatornamewithlimits{inf}_{c_i} \left \| f - \sum_{i=1}^{n}
c_i u_i \right\|_p, \quad n = 1, 2, \dots,
$$
are called the {\it Kolmogorov widths} of $\bF$ in $L_p$. In the definition of
the Kolmogorov widths we take for $f\in \bF$, as an approximating element
from $U := \sp \{u_i \}_{i=1}^n$ the element of best
approximation. This means
that in general (i.e. if $p\neq 2$) this method of approximation is not linear.

We have the following obvious inequalities
\be\label{I1}
d_m (\bF, L_p)\le \varrho_m^*(\bF,L_p)\le \varrho_m(\bF,L_p).
\ee

In this paper we consider the case $p=2$, i.e. recovery takes place in the Hilbert space $L_2$.
The main result of the paper is the following general inequality.
\begin{Theorem}\label{BT1} Let $\bF$ be a compact subset of $\C(\Omega)$. There exist two positive absolute constants $b$ and $B$ such that
$$
\ro_{bn}(\bF,L_2) \le Bd_n(\bF,L_\infty).
$$
\end{Theorem}
Note that for special sets $\bF$ (in the reproducing kernel Hilbert space setting) the following inequality is known (see \cite{NSU} and \cite{KU})
$$
\ro_{n}(\bF,L_2) \le C\left(\frac{\log n}{n}\sum_{k\ge cn} d_k (\bF, L_2)^2\right)^{1/2}
$$
with universal constants $C,c>0$. 

We discuss in Section \ref{B} applications of Theorem \ref{BT1} to several classes of functions with mixed smoothness. 

\section{Conditional result}
\label{A}

Let $X_N$ be an $N$-dimensional subspace of the space of continuous functions $\C(\Omega)$. For a fixed $m$ and a set of points  $\xi:=\{\xi^\nu\}_{\nu=1}^m\subset \Omega$ we associate with a function $f\in \C(\Omega)$ a vector
$$
S(f,\xi) := (f(\xi^1),\dots,f(\xi^m)) \in \bbC^m.
$$
Denote
$$
\|S(f,\xi)\|_p:= \left(\frac{1}{m}\sum_{\nu=1}^m |f(\xi^\nu)|^p\right)^{1/p},\quad 1\le p<\infty,
$$
and 
$$
\|S(f,\xi)\|_\infty := \max_{\nu}|f(\xi^\nu)|.
$$
For a positive weight $\bw:=(w_1,\dots,w_m)\in \R^m$ consider the following norm
$$
\|S(f,\xi)\|_{p,\bw}:= \left(\sum_{\nu=1}^m w_\nu |f(\xi^\nu)|^p\right)^{1/p},\quad 1\le p<\infty.
$$
Define the best approximation of $f\in L_p(\Omega,\mu)$, $1\le p\le \infty$ by elements of $X_N$ as follows
$$
d(f,X_N)_p := \inf_{u\in X_N} \|f-u\|_p.
$$
It is well known that there exists an element, which we denote $P_{X_N,p}(f)\in X_N$, such that
$$
\|f-P_{X_N,p}(f)\|_p = d(f,X_N)_p.
$$
The operator $P_{X_N,p}: L_p(\Omega,\mu) \to X_N$ is called the Chebyshev projection. 

We will prove Theorem \ref{AT1} below under the following assumptions.

{\bf A1. Discretization.} Suppose that $\xi:=\{\xi^j\}_{j=1}^m\subset \Omega$ is such that for any 
$u\in X_N$ we have
$$
C_1\|u\|_p \le \|S(u,\xi)\|_{p,\bw}  
$$
with a positive constant $C_1$ which may depend on $d$ and $p$. 

{\bf A2. Weight.} Suppose that there is a positive constant $C_2=C_2(d,p)$ such that 
$\sum_{\nu=1}^m w_\nu \le C_2$.

Consider the following well known recovery operator (algorithm) (see, for instance, \cite{CM})
$$
\ell p\bw(\xi)(f) := \ell p\bw(\xi,X_N)(f):=\text{arg}\min_{u\in X_N} \|S(f-u,\xi)\|_{p,\bw}.
$$
Note that the above algorithm $\ell p\bw(\xi)$ only uses the function values $f(\xi^\nu)$, $\nu=1,\dots,m$. In the case $p=2$ it is a linear algorithm -- orthogonal projection with respect 
to the norm $\|\cdot\|_{2,\bw}$. Therefore, in the case $p=2$ approximation error by the algorithm $\ell 2\bw(\xi)$ gives an upper bound for the recovery characteristic $\ro_m(\cdot, L_2)$. In the case $p\neq 2$ approximation error by the algorithm $\ell p\bw(\xi)$ gives an upper bound for the recovery characteristic $\ro_m^*(\cdot, L_p)$.

\begin{Theorem}\label{AT1} Under assumptions {\bf A1} and {\bf A2} for any $f\in \C(\Omega)$ we have
$$
\|f-\ell p\bw(\xi)(f)\|_p \le (2C_1^{-1}C_2^{1/p} +1)d(f, X_N)_\infty.
$$
\end{Theorem}
\begin{proof}
 From the definition of the operator $P_{X_N,\infty}$ we obtain  
\be\label{A1}
\|f-P_{X_N,\infty}(f)\|_p \le \|f-P_{X_N,\infty}(f)\|_\infty = d(f, X_N)_\infty.
\ee
Clearly,
$$
\|S(f-P_{X_N,\infty}(f),\xi)\|_\infty \le \|f-P_{X_N,\infty}(f)\|_\infty = d(f, X_N)_\infty.
$$
Therefore, by {\bf A2} we get
\be\label{A2}
\|S(f-P_{X_N,\infty}(f),\xi)\|_{p,\bw} \le C_2^{1/p} \|S(f-P_{X_N,\infty}(f),\xi)\|_\infty \le C_2^{1/p}d(f, X_N)_\infty.
\ee
Next, by the definition of the algorithm $\ell p\bw(\xi)$ and by {\bf A2} we obtain
\be\label{A3}
\|S(f-\ell p\bw(\xi)(f),\xi)\|_{p,\bw} \le \|S(f-P_{X_N,\infty}(f),\xi)\|_{p,\bw} \le C_2^{1/p}d(f, X_N)_\infty.
\ee
Bounds (\ref{A2}) and (\ref{A3}) imply
\be\label{A4}
\|S(P_{X_N,\infty}(f)-\ell p\bw(\xi)(f),\xi)\|_{p,\bw} \le 2C_2^{1/p}d(f, X_N)_\infty.
\ee
Then, the discretization assumption {\bf A1} implies
\be\label{A5}
\|P_{X_N,\infty}(f)-\ell p\bw(\xi)(f)\|_{p} \le C_1^{-1} 2C_2^{1/p}d(f, X_N)_\infty.
\ee
Combining bounds (\ref{A1}) and (\ref{A5}) we conclude
$$
\|f-\ell p\bw(\xi)(f)\|_p \le (1+2C_1^{-1} C_2^{1/p})d(f, X_N)_\infty,
$$
which completes the proof of Theorem \ref{AT1}.
\end{proof}

\section{Applications}
\label{B}

{\bf Proof of Theorem \ref{BT1}.} Let $X_n$ be a subspace of dimension $n$ satisfying: for all $f\in \bF$
\be\label{B1}
d(f,X_n)_\infty \le 2d_n(\bF,L_\infty).
\ee
We now use a result on discretization in $L_2$ from \cite{LT} (see Theorem 3.3 there), which is a generalization to the complex case of an earlier result from \cite{DPSTT2} established for the real case. 

\begin{Theorem}\label{BT2} If $X_N$ is an $N$-dimensional subspace of the complex $L_2(\Omega,\mu)$, then there exist three absolute positive constants $C_1'$, $c_0'$, $C_0'$,  a set of $m\leq   C_1'N$ points $\xi^1,\ldots, \xi^m\in\Omega$, and a set of nonnegative  weights $\lambda_j$, $j=1,\ldots, m$,  such that
\[ c_0'\|f\|_2^2\leq  \sum_{j=1}^m \lambda_j |f(\xi^j)|^2 \leq  C_0' \|f\|_2^2,\  \ \forall f\in X_N.\]
\end{Theorem}

For our application we need to satisfy the assumption {\bf A2} on weights. 
\begin{Remark}\label{BR1} Considering a new subspace $X_N' := \{f\,:\, f= g+c, \, g\in X_N,\, c\in \bbC\}$ and applying Theorem \ref{BT2} to 
the $X_N'$ with $f=1$ ($g=0$, $c=1$) we conclude that a version of Theorem \ref{BT2} holds with $m\le C_1'N$ replaced by $m\le C_1'(N+1)$ and with weights satisfying 
$$
\sum_{j=1}^m \lambda_j \le C_0'.
$$
\end{Remark}

Let now $\xi=\{\xi^\nu\}_{\nu=1}^m$ be the set of points from Theorem \ref{BT2} and Remark \ref{BR1} with $X_N = X_n$. Then $m\le bn$ and assumptions {\bf A1} and {\bf A2} are satisfied with absolute constants $C_i$, $i=1,2$. Applying Theorem \ref{AT1} we complete the proof 
of Theorem \ref{BT1}.

%\end{proof}
We now proceed to applications of Theorem \ref{BT1} for classes of functions of mixed smoothness.
We define the class $\bW^r_{q}$ in the following way. For $r > 0$   the functions
$$
F_r(x):=1+2\sum_{k=1}^{\infty}k^{-r}\cos(kx - r\pi/2)
$$
are called {\it Bernoulli kernels}.
 Let  
$$
F_r (\bx) := \prod_{j=1}^d F_r (x_j)
$$
be the multivariate analog of the Bernoulli kernel. 
We denote by $\bW_{q}^r$ the class of functions
$f(\bx)$ representable in the form
$$
f(\bx) =\varphi(\bx)\ast F_r (\bx) :=
(2\pi)^{-d}\int_{\T^d}\varphi(\by)F_r (\bx-\by)d\by,
$$
where $\varphi\in L_q$ and $\|\varphi\|_q\le 1$. In this case the function
$\varphi$ is called
$r$-derivative of $f$ and is denoted by
$\varphi(\bx) = f^{(r)}(\bx)$.
Note that in the case of integer $r$ the class $\bW^r_{q}$   is equivalent to the class defined by restrictions on mixed
derivatives. 

{\bf 1. Recovery of $\bW^r_2$.} The following upper bound is known (see \cite{TrBe})  \be\label{B2}
d_n(\bW^r_{2}, L_\infty) \le C(r,d) n^{-r}(\log n)^{(d-1)r+1/2},\quad    r>1/2.
\ee
 By Theorem \ref{BT1} we obtain from bound (\ref{B2})  the estimate
\be\label{B3}
\ro_n(\bW^r_{2}, L_2) \le C'(r,d) n^{-r}(\log n)^{(d-1)r+1/2}, \quad r>1/2.
\ee
Very recently bound (\ref{B3}) was obtained in \cite{NSU}. This is the best known upper bound.
For the previous breakthrough result see \cite{KU}. Thus, we demonstrate here that  Theorem \ref{BT1} is a rather powerful tool in estimation of the recovery numbers. The right order of the quantity $\ro_n(\bW^r_{2}, L_2)$ is not known.
The reader can find related results in \cite{VTbookMA}, Ch.6, \cite{DTU}, Ch.5, and \cite{NSU}.

{\bf 2. Recovery of $\bW^r_1$.} We define the best $n$-term approximation with respect to the trigonometric system $\Tr^d:=\{e^{i(\bk,\bx)}\}_{\bk\in \Z^d}$ as follows
$$
\sigma_n(f)_p := \inf_{\bk^1,\dots,\bk^n}\inf_{c_1,\dots,c_n}\left\|f-\sum_{j=1}^n c_je^{i(\bk^j,\bx)}\right\|_p.
$$
The following result is known (see, for instance, \cite{VTbookMA}, p.466)
\be\label{B4}
\sigma_n(F_{r})_\infty \le C(r,d)n^{-r+1/2}(\log n)^{r(d-1)+1/2},\quad r>1.
\ee
Bound (\ref{B4}) and the definition of the class $\bW^r_{1}$ imply that there exists 
a set of frequencies $\Lambda_n =\{\bk^j\}_{j=1}^n$ such that for any $f\in \bW^r_{1}$
we have
\be\label{B5}
d(f,\Tr(\Lambda_n))_\infty \le C(r,d)n^{-r+1/2}(\log n)^{r(d-1)+1/2},
\ee
where we use the notation
$$
\Tr(\Lambda):=\{f\,:\, f=\sum_{\bk\in\Lambda} c_\bk e^{i(\bk,\bx)}\}.
$$
We now need a discretization result from \cite{VT158} (see Theorem 1.1 there).

\begin{Theorem}\label{BT3} There are three positive absolute constants $C_1$, $C_2$, and $C_3$ with the following properties: For any $d\in \N$ and any $Q\subset \Z^d$   there exists a set of  $m \le C_1|Q| $ points $\xi^j\in \T^d$, $j=1,\dots,m$, such that for any $f\in \Tr(Q)$ 
we have
$$
C_2\|f\|_2^2 \le \frac{1}{m}\sum_{j=1}^m |f(\xi^j)|^2 \le C_3\|f\|_2^2.
$$
\end{Theorem}

Therefore, conditions {\bf A1} and {\bf A2} are satisfied for the $\Tr(\Lambda_n)$ and by Theorem \ref{AT1} we obtain from  (\ref{B5}) 
\be\label{B6}
\ro_n(\bW^r_{1},L_2) \le C(r,d)n^{-r+1/2}(\log n)^{r(d-1)+1/2},\quad r>1.
\ee
Moreover, as a recovering algorithm we can take the $\ell2\bw_n(\xi)$ with $\bw_n=(1/n,\dots,1/n)$, which is a standard least squares algorithm. 

{\bf Recovery of $\bH^r_2$.} We now turn our discussion to the classes $\bH^r_2$.
Let  $\btt =(t_1,\dots,t_d )$ and $\Delta_{\btt}^l f(\bx)$
be the mixed $l$-th difference with
step $t_j$ in the variable $x_j$, that is
$$
\Delta_{\btt}^l f(\bx) :=\Delta_{t_d}^l\dots\Delta_{t_1}^l
f(x_1,\dots ,x_d ) .
$$
Let $e$ be a subset of natural numbers in $[1,d ]$. We denote
$$
\Delta_{\btt}^l (e) =\prod_{j\in e}\Delta_{t_j}^l,\qquad
\Delta_{\btt}^l (\varnothing) = I .
$$
We define the class $\bH_{q,l}^r B$, $l > r$, as the set of
$f\in L_q$ such that for any $e$
\be\label{B7}
\bigl\|\Delta_{\btt}^l(e)f(\bx)\bigr\|_q\le B
\prod_{j\in e} |t_j |^r .
\ee
In the case $B=1$ we omit it. It is known (see, for instance, \cite{VTbookMA}, p.137) that the classes $\bH^r_{q,l}$ with different $l$ are equivalent. So, for convenience we fix one $l= [r]+1$ and omit $l$ from the notation. The following bound for the Kolmogorov width is known (see 
\cite{Bel})
\be\label{B7}
d_n(\bH^r_{2}, L_\infty) \le C(r,d) n^{-r}(\log n)^{(d-1)(r+1/2)+1/2},\quad    r>1/2.
\ee
We obtain from bound (\ref{B7}) by Theorem \ref{BT1} the estimate
\be\label{B8}
\ro_n(\bH^r_{2}, L_2) \le C'(r,d) n^{-r}(\log n)^{(d-1)(r+1/2)+1/2}, \quad r>1/2.
\ee
Let us make a brief historical comment on optimal recovery of classes $\bH^r_p$ in $L_p$. For more detailed discussion we refer the reader to \cite{DTU} and \cite{VTbookMA}. The first result 
 in this direction was established in \cite{Tem8}
$$
\ro_n(\bH^r_{p}, L_p) \le C(r,d,p) n^{-r}(\log n)^{(d-1)(r+1)}, \quad r>1/p,\quad 1\le p\le \infty.
$$
We note that the problem of the right asymptotic behavior of $\ro_n(\bH^r_{p}, L_p)$, $1\le p\le \infty$, is a great open problem. As far as we know it is only solved in the case $d=2$, $p=\infty$
(see, for instance, \cite{VTbookMA}, p.308):
$$
\ro_n(\bH^r_{\infty}, L_\infty) \asymp n^{-r}(\log n)^{r+1}.
$$

\section{Discussion}
\label{C}

The discretization Theorem \ref{BT2} plays a key role in the proof of the main result of the paper -- Theorem \ref{BT1}. We would like to extend Theorem \ref{BT1} from recovery in $L_2$ to recovery in $L_p$, $1\le p<\infty$. Theorem \ref{AT1} provides the required bound for the 
algorithm $\ell p \bw(\xi)$ for all $1\le p<\infty$. However, we do not have an analog of Theorem 
\ref{BT2} for $p\neq 2$. For the reader's convenience we present here some relevant discretization results. The following result is from \cite{DPSTT2}.

\begin{Theorem}\label{CT1} Given  $1\leq p\leq 2$,      an arbitrary  $N$-dimensional subspace $X_N$  of $L_p(\Omega,\mu)$ and any $\va\in (0, 1)$,  there exist   $\xi^1,\ldots, \xi^m\in\Omega$ and   $w_1,\ldots, w_m>0$  such that  $m\leq C_p(\va) N\log^3 N$ 	 and
	\begin{align}\label{C1}
	(1-\va)	\|f\|_{p} \leq   \left(\sum_{\nu=1}^m w_\nu |f(\xi^\nu)|^p\right)^{\frac1p}\leq (1+\va) \|f\|_{p},\   \  \forall f\in X_N.
	\end{align}	
\end{Theorem}
An important good feature of Theorem \ref{CT1} is that it applies to any subspace. However,
here is a reason why a combination of Theorem \ref{CT1} (instead of Theorem \ref{BT2}) and Theorem \ref{AT1} does not give a new result. There is a restriction $p\le 2$ in Theorem \ref{CT1}. It is well known that $\|\cdot\|_p \le \|\cdot\|_2$ provided $p\le 2$. Therefore, 
Theorem \ref{BT1} covers the case $1\le p<2$ as well. Next, a version of Theorem \ref{CT1}
for $p>2$ would give a new result
\be\label{C2}
\ro_{bn(\log n)^3}^*(\bF,L_p) \le Bd_n(\bF,L_\infty)
\ee
with $b$ and $B$ allowed to depend on $p$. However, there is no known analog of Theorem \ref{CT1} for $2<p<\infty$. 

We may want to have the recovery algorithm $\ell 2 \bw(\xi)$ to be a classical least squares 
algorithm, i.e. $\bw=\bw_m:=(1/m,\dots,1/m)$. For that we need an analog of the discretization Theorem
\ref{BT2} with the weight $\bw_m $. There is such an analog of Theorem \ref{BT2} 
but under an extra assumption on the subspace $X_N$. First, we formulate the corresponding theorem from \cite{LT} and then we give the definition of Condition E($t$).    
 \begin{Theorem}\label{CT2} Let  $\Omega\subset \R^d$ be a compact set with the probability measure $\mu$. Assume that $\{u_i(\bx)\}_{i=1}^N$ is a real (or complex) orthonormal system in $L_2(\Omega,\mu)$ satisfying Condition E($t$). 
Then there is an absolute  constant $C_1$ such that there exists a set $\{\xi^j\}_{j=1}^m\subset \Omega$ of $m \le C_1 t^2 N$ points with the property:
 For any $f=\sum_{i=1}^N c_iu_i$  we have  
\begin{equation*}
C_2 \|f\|_2^2 \le \frac{1}{m}\sum_{j=1}^m |f(\xi^j)|^2 \le C_3 t^2\|f\|_2^2, 
\end{equation*}
where $C_2$ and $C_3$ are absolute positive constants. 
\end{Theorem}

{\bf Condition E($t$).} We say that an orthonormal system $\{u_i(\bx)\}_{i=1}^N$ defined on $\Omega$ satisfies Condition E($t$) with a constant $t$ if for all $\bx\in \Omega$
%\be\label{ud5}
$$
 \sum_{i=1}^N |u_i(\bx)|^2 \le Nt^2.
$$

Theorem \ref{CT2} combined with Theorem \ref{AT1} gives the following analog of Theorem \ref{BT1}. We need some definitions for its formulation. For a fixed $m$ and a set of points  $\xi:=\{\xi^j\}_{j=1}^m\subset \Omega$ denote for a class $\bF$ (usually, a centrally symmetric compact in $L_2(\Omega,\mu)$)
$$
\varrho_m^{ls}(\bF,L_2) := \inf_{\xi,\,X_N} \sup_{f\in \bF} \|f-\ell 2 \bw_m(\xi,X_N)(f)\|_2.
$$
We now define $E(t)$-conditioned Kolmogorov width
$$
d_N^{E(t)}(\bF,L_p) :=     \inf_{\{u_1,\dots,u_N\}\, \text{satisfies Condition} E(t)}   \sup_{f\in \bF}\inf_{c_1,\dots,c_N}\|f-\sum_{i=1}^N c_iu_i\|_p.
$$

\begin{Theorem}\label{CT3} Let $\bF$ be a compact subset of $\C(\Omega)$. There exist two positive   constants $b$ and $B$ which may depend on $t$ such that
$$
\ro_{bn}^{ls}(\bF,L_2) \le Bd_n^{E(t)}(\bF,L_\infty).
$$
\end{Theorem}

We have discussed possible applications of two results from discretization -- Theorems \ref{CT1} and \ref{CT2}. The reader can find other recent results on discretization in \cite{DPTT}, \cite{DPSTT1}, \cite{VT168}, \cite{Kos}, and \cite{VT159}. 

{\bf Acknowledgements.} The author is grateful to Irina Limonova for useful comments and discussions. 

The work was supported by the Russian Federation Government Grant N{\textsuperscript{\underline{o}}}14.W03.31.0031. The paper contains results obtained in frames of the program \lq\lq Center for the storage and analysis of big data", supported by the Ministry of Science and High Education of Russian Federation (contract 11.12.2018 N{\textsuperscript{\underline{o}}}13/1251/2018 between the Lomonosov Moscow State University and the Fund of support of the National technological initiative projects).

\end{document}